\documentclass[reqno,final,12pt]{amsart}
\usepackage{amssymb,amsthm,amsmath,color,fullpage,url}
\usepackage{comment}
\usepackage{hyperref}
\usepackage[noabbrev,capitalise]{cleveref}
\usepackage{thm-restate}

\newtheorem{thm}{Theorem}[section]
\newtheorem{lem}[thm]{Lemma}

\newtheorem{cor}[thm]{Corollary}

\newcommand{\eps}{\varepsilon}
 
\newcommand{\mc}[1]{\mathcal{#1}} 
\newcommand{\bb}[1]{\mathbb{#1}}
\newcommand{\brm}[1]{\operatorname{#1}}

\newcommand{\bs}[1]{\boldsymbol{#1}}
\newcommand{\stww}{\operatorname{stww}}
\newcommand{\tww}{\operatorname{tww}}
\newcommand{\mad}{\operatorname{mad}}
\newcommand{\fs}[2]{\left(\frac{#1}{#2}\right)}
\newcommand{\s}[1]{\left(#1\right)}

\usepackage[noabbrev,capitalise]{cleveref}
\renewcommand{\v}{\textup{\textsf{v}}}
\newcommand{\e}{\textup{\textsf{e}}}

\newcommand{\N}{\mathbb{N}}

\begin{document}

\title{Twin-width of sparse random graphs}

\thanks{The first author is supported by the Institute for Basic Science (IBS-R029-C1). The second and fourth authors are supported by the Natural Sciences and Engineering Research Council of Canada (NSERC). Les deuxi\`eme et quatri\`eme auteurs sont supportés par le Conseil de recherches en sciences naturelles et en génie du Canada (CRSNG). The third author is funded by the SNSF Ambizione Grant No.~216071.}

\subjclass[2020]{}

\author{Kevin Hendrey}
\address{Discrete Mathematics Group, Institute for Basic Science (IBS), Daejeon, South Korea}
\email{kevinhendrey@ibs.re.kr}
\urladdr{sites.google.com/view/kevinhendrey/}

\author{Sergey Norin}
\address{Department of Mathematics and Statistics, McGill University, Montr\'eal, Canada}
\email{snorin@math.mcgill.ca}
\urladdr{www.math.mcgill.ca/snorin/}

\author{Raphael Steiner}
\address{Institute of Theoretical Computer Science, Department of Computer Science, ETH Z\"{u}rich, Switzerland}
\email{raphaelmario.steiner@inf.ethz.ch}
\urladdr{sites.google.com/view/raphael-mario-steiner/}

\author{J\'er\'emie Turcotte}
\address{Department of Mathematics and Statistics, McGill University, Montr\'eal, Canada}
\email{mail@jeremieturcotte.com}
\urladdr{www.jeremieturcotte.com}

	\begin{abstract}
		We show that the twin-width of every $n$-vertex $d$-regular graph is at most $n^{\frac{d-2}{2d-2}+o(1)}$ and that almost all $d$-regular graphs attain this bound. More generally, we  obtain  bounds on the twin-width of sparse Erd\H{o}s-Renyi  and regular random graphs, complementing the bounds in the denser regime due to Ahn, Chakraborti, Hendrey, Kim and Oum.
  \end{abstract}
	\maketitle

\section{Introduction}

The twin-width $\tww(G)$ of a graph $G$ is a parameter recently introduced by Bonnet, Kim, Thomass\'{e} and Watrigant~\cite{bonnet_twin-width_2022}.\footnote{Informally, \emph{twin-width} of a graph $G$ is the minimum integer $w$  such that there exists a sequence of $|V(G)| - 1$ iterated vertex identifications, starting with $G$ such that at every step of the process, every vertex is incident with at most $w$ red edges, where a red edge appears between two sets of identified vertices if they are neither fully adjacent nor fully non-adjacent. We will not need the formal definition in this paper, and will be working with a simpler notion of \emph{sparse twin-width} which we introduce at the end of this section.} Twin-width has received a lot of attention recently largely due to its algorithmic applications, as it is shown in~\cite{bonnet_twin-width_2022} that any first-order formula can be evaluated in linear time on graphs in a given class of bounded twin-width, given a witness certifying that twin-width is bounded. Structured graph classes such as proper minor closed graph classes and classes of bounded clique-width or rank-width have bounded twin-width. Meanwhile, a counting argument can be used to show that there exist $3$-regular graphs of unbounded twin-width~\cite{bonnet_twin-width_2022-1}.

Bounds on the twin-width of general graphs were first considered by  Ahn, Hendrey, Kim and Oum~\cite{ahn_bounds_2022}, who have shown that $\tww(G) \leq \s{\frac{1}{2}+o(1)}n$ and that \begin{equation}\label{e:mUpper} \tww(G) \leq \s{1 + o(1)} \sqrt{3m} \end{equation} for every graph $G$ with $n$ vertices and $m$ edges. They have also started investigating the twin-width of random graphs $G(n,p)$. This investigation was continued by Ahn et al.~\cite{ahn_twin-width_2022} where it was shown that with high probability the twin-width of $G(n,p)$ is $\Theta(n \sqrt{p})$ for $726\log n/n \leq p \leq 1/2$, i.e. \eqref{e:mUpper} is tight up to a constant factor in this regime. Meanwhile, already in~\cite{ahn_bounds_2022}  it was shown that if $p < c/n$ for fixed $c<1$ then the twin-width of $G(n,p)$ is at most two. These results left open the question of determining the typical twin-width of $G(n,p)$ for the range $1/n < p =  o(\log n/n )$. As one of the contributions of this paper, we almost fully resolve this question (see Theorem~\ref{t:main}).

We also consider the twin-width of random $d$-regular $n$-vertex graphs for $d =  o (\log n)$. Prior to our work, very little was known even in the case $d=3$. As mentioned above, $3$-regular graphs have unbounded twin-width, but the counting argument only gives a lower bound which is triply logarithmic in the number of vertices. Meanwhile, the best prior upper bound on the twin-width of $n$-vertex $3$-regular graphs was of the order $O(\sqrt{n})$ as above. In this paper, we precisely determine the maximum twin-width among all $n$-vertex $3$-regular graphs as $n^{\frac{1}{4}+o(1)}$, and show that $\tww(G) =  n^{1/4 + o(1)}$ for almost all $n$-vertex $3$-regular graphs $G$. These results generalize to $d$-regular graphs, determining precisely the asymptotics for the maximum twin-width of $n$-vertex $d$-regular graphs for any $d$.

\begin{thm}\label{t:regular}
	If $d \geq 2$ is an integer, then $$  \tww(G) \leq   n^{\frac{d-2}{2d-2}+ o(1)} $$ for every $n$-vertex $d$-regular graph $G$. Moreover, the inequality holds with equality for almost all $n$-vertex $d$-regular graphs.
\end{thm}

More generally, for various sparse random graph models we show that the twin-width is typically polynomial in the number of vertices, with the exponent governed by the maximum average degree of the graph. 

Recall that 
the \emph{maximum average degree} of a graph $G$, denoted by $\mad(G)$, is defined as the maximum average degree of subgraphs of $G$, i.e. $$\mad(G) = \max_{H \subseteq G, \v(H)>0}\frac{2\e(H)}{\v(H)}.$$  
(We denote the number of vertices and edges of a graph $H$ by $\v(H)$ and $\e(H)$, respectively.)

Let us  first state a special case of our most general result for 
Erd\H{o}s-Renyi random graphs with constant average degree.  

\begin{thm}\label{t:ER}
	For every $c > 1$ there exists $\rho = \rho(c) > 2c$ such that with high probability $$ \mad\s{G\s{n,\frac{c}{n}}}=(1+o(1))\rho \qquad \mathrm{and} \qquad \tww\s{G\s{n,\frac{c}{n}}} =   n^{\frac{\rho-2}{2\rho-2}+ o(1)}. $$
\end{thm}

Our most general result gives estimates of the typical twin-width for several models of sparse random graphs with average degree bounded from below away from one. It extends Theorems~\ref{t:regular} and~\ref{t:ER} as well as giving tighter bounds on the error terms.

\begin{thm}\label{t:main} Fix $\eps > 0$. Let $G$ be a random $n$-vertex graph obtained using any of the following models:
	\begin{itemize}
		\item $G$ is a uniformly random $d$-regular graph on $n$ vertices with $2 < d < \frac{1}{5}\log n$,
		\item $G$ is a uniformly random graph on $n$ vertices and $m$ edges with $(1+\eps)n \leq  m \leq \frac{1}{11}n\log n$, or
		\item $G$ is a random graph $G(n,p)$ with $\frac{1+\eps}{n} \leq p \leq  \frac{1}{12}\frac{\log n}{n}$.
	\end{itemize}
Then, with high probability, we have\begin{equation}\label{e:main}
  n^{\frac{d-2}{2d-2} }\cdot \Omega\s{ \frac{d}{\sqrt{\log n}}} \leq \tww(G) \leq n^{\frac{d-2}{2d-2}} \cdot e^{O\s{d^{-1}\sqrt{\log n}\cdot {\log \log n} +\log d}},
\end{equation}
	where $d = \mad(G)$. (The constants hidden in the $\Omega(\cdot)$ and $O(\cdot)$ notations depend only on $\eps$.) 
\end{thm}

Our results, in particular, answer a question of 
Sylvester~\cite{sylvester_open_nodate} whether  $\tww(G)=\Theta(\sqrt{\e(G)})$ with high probability for sparse Erd\H{o}s-Renyi  random graphs and random cubic graphs. By \cref{t:main} the answer is negative for random $d$-regular graphs with   $d=o(\frac{\log n}{n \log \log n})$ and for $G(n,p)$ with $p = o(\frac{\log n}{n \log \log n})$ and  the answer is positive for $p = \Omega(\frac{\log n}{n})$. 

It follows from the results of~\cite{ahn_bounds_2022,ahn_twin-width_2022} that $\log\tww(G(n,p)) = o(\log n)$ when $p = \frac{1-o(1)}{n}$ and $\log\tww(G(n,p) = \frac{1}{2} - o(1)$ for $p = \Omega(\frac{\log n}{n})$. \cref{t:main} establishes the asymptotics of $\log\tww(G(n,p)$ and $\frac{1}{2}-\log\tww(G(n,p)$ in most of the intermediate range, i.e. for $ \frac{1+o(1)}{n} \leq p \leq o(\frac{\log n}{n \log \log n})$

We finish the introduction by briefly and informally describing the techniques we use to prove Theorems~\ref{t:regular},~\ref{t:ER} and~\ref{t:main}.  
To obtain the upper bounds on twin-width of a graph $G$, we map the vertices of $G$ to random sequences over some alphabet so that the images of far away vertices are independent and uniformly distributed over the set of all sequences, while the images of neighbours share a fair amount of information.
This yields a homomorphism from the original graph to a structured and appropriately sparse graph $H$, allowing us to bound the twin-width of $G$ in terms of the twin-width of $H$. We then finish by bounding the twin-width of the more structured graph $H$. We prove the upper bounds in \cref{s:upper}.

Our proofs of lower bounds on the twin-width rely on counting, which as far as we know remains the only successful technique for establishing superconstant lower bounds on the twin-width of graphs with constant maximum degree. As a new ingredient to facilitate counting arguments, we show that for every graph $G$ with twin-width at most $w$ and any $K \leq \v(G)$, $G$ is a spanning subgraph of a ``fairly balanced'' blowup of some $(w+\Delta(G))$-degenerate graph on $K$ vertices. We prove the lower bounds in \cref{s:lower}.

\subsection*{Preliminaries, basic definitions and notation} We mostly use standard graph theoretical notation. Let $G$ be a graph and $X\subseteq V(G)$. We denote by $\Delta(G)$ the maximum degree of $G$, by $G[X]$ the subgraph of $G$ induced by $X$, and by $G\setminus X$ the graph $G[V(G)\setminus X]$ (if $X=\{v\}$, we write $G \setminus v$).

Let $[n] = \{1,2,\ldots,n\}$. If $S$ and $T$ are sets, $\bs{v}=(s_1,s_2,\ldots, s_r) \in S^r$ is a sequence of elements of $S$, and $f: S \to T$ is  map, let $f(\bs{v}):=(f(s_1),f(s_2),\ldots,f( s_r)) \in T^r.$ 

It will be convenient for us to work with a somewhat atypical notion of homomorphism, where we allow adjacent vertices to be mapped to the same vertex. Explicitly, for graphs $G$ and $H$ we say that a map $\phi : V(G) \to V(H)$ is a \emph{homomorphism from $G$ to $H$} if for every $uv \in E(G)$ we either have $\phi(u)\phi(v) \in E(H)$  or $\phi(u) = \phi(v)$.

Next let us elaborate on the relationship between  Theorems~\ref{t:regular},~\ref{t:ER} and~\ref{t:main}. The error terms in \eqref{e:main} are of the form $n^{o(1)}$ in the setting covered by the theorem and so  Theorem~\ref{t:main} implies Theorems~\ref{t:regular} and~\ref{t:ER} except \begin{itemize}
	\item the upper bound in Theorem~\ref{t:regular} needs to hold for all $n$-vertex $d$-regular graphs  not just almost all such graphs, and
	\item Theorem~\ref{t:ER} postulates concentration of $\mad(G(n,p))$ for $p=c/n$ and constant $c$, i.e. the existence of $\rho(c)$ such that $\mad(G(n,c/n)) = (1+o(1))\rho(c)$ with high probability. Anantharam and Salez establish existence of such a $\rho$ in~\cite{anantharam_densest_2016} and show that $\rho$ can be found as a fixed point of a certain distributional fixed-point equation, as was conjectured earlier by Hajek~\cite{hajek_performance_1990}. Thus by the results from~\cite{anantharam_densest_2016} Theorem~\ref{t:ER} is implied by Theorem~\ref{t:main}.
\end{itemize}

Finally, it is well-known that the random graph $G(n,p)$ with high probability has $m = p\binom{n}{2} + o(\sqrt{p}n \log n)$ edges, and for fixed $m$ every graph with $V(G)=[n]$ and $\e(G)=m$ has the same probability in $G(n,p)$. Thus if \eqref{e:main} holds for any value of $m = p\binom{n}{2} \pm o(\sqrt{p}n \log n)$ with high probability (independent of the choice of $m$), then it holds with high probability for $G(n,p)$. It follows that the conclusion of Theorem~\ref{t:main} for $G(n,p)$ follows from its conclusion for uniformly random graphs on $n$ vertices and $m$ edges. 

Thus Theorems~\ref{t:regular},~\ref{t:ER} and~\ref{t:main} follow from the next two theorems, which we prove in  Sections~\ref{s:upper} and~\ref{s:lower}, respectively.

\begin{restatable}[Upper bounds]{thm}{upper}\label{t:upper}\item
	\begin{itemize}
		\item[(i)] If $G$ is an $n$-vertex, $d$-regular graph with  $d > 2$, then 
\begin{equation}\label{e:mainUpper}	\tww(G) \leq n^{\frac{d-2}{2d-2}} \cdot e^{O\s{d^{-1}\sqrt{\log n}\cdot {\log \log n} +\log d}}. \end{equation}
		\item[(ii)] Let $\eps > 0$ be fixed. If $G$ is a uniformly random graph on $n$ vertices and $m$ edges with $m \geq (1+\eps)n$ and $d = \mad(G)$, then \eqref{e:mainUpper} holds with high probability.
	\end{itemize}	
\end{restatable}

\begin{restatable}[Lower bounds]{thm}{tlower}\label{t:lower} 
	If $G$ is either
	\begin{itemize}
		\item[(i)] a uniformly random $n$-vertex, $d$-regular graph with $2 < d < \frac{1}{5}\log n$, or 
		\item[(ii)] a uniformly random graph on $n$ vertices and $m$ edges where $(1+\eps)n \leq  m \leq \frac{1}{11}n\log n$ and $d = \mad(G)$,
	\end{itemize}	
		then with high probability		
			\begin{equation}\label{e:lower}	\tww(G) \geq
		n^{\frac{d-2}{2d-2} }\cdot \Omega\s{ \frac{d}{\sqrt{\log n}}}.  		
		\end{equation}	
\end{restatable}

The time has come to define twin-width, or rather a substitute of it that we will work with. A \emph{contraction sequence from a graph $G$ to a graph $H$} is a sequence of graphs $(G=G_1,G_2,\ldots, G_k=H)$ such that for each $i \in [k-1]$ the graph $G_{i+1}$ is obtained from $G_i$ by identifying two vertices. That is, for some pair of distinct vertices $u,v \in G_i$ the graph $G_{i+1}$ is obtained from $G_i \setminus \{u , v\}$ by adding a new vertex $w$ such that  $wx \in E(G_{i+1})$ if and only if $ux \in E(G_i)$ or $vx \in E(G_i)$. We say that a contraction sequence $(G_1,G_2,\ldots, G_k)$ is a \emph{$w$-contraction sequence} for some positive integer $w$ if $\Delta(G_i) \leq w$ for every $i \in [k]$, and this sequence is a \emph{contraction sequence for the graph $G_1$} if $\v(G_k)=1$.
Finally, we define the \emph{sparse twin-width of a graph $G$}, denotes by $\stww(G)$, as the minimum positive integer $w$ such that there exists a $w$-contraction sequence for $G$. That is, $\stww(G)$ is the  minimum $w$ such that we can reduce $G$ to a one-vertex graph via a sequence of pairwise vertex identifications such that the degrees of all the vertices throughout this process are at most $w$.

The twin-width $\tww(G)$ is defined similarly to the sparse twin-width, but using a notion of trigraphs. For every graph $G$, we have
$$ \max\{\tww(G),\Delta(G)\} \leq \stww(G) \leq \tww(G) + \Delta(G).$$
The maximum degree $\Delta(G)$ is much smaller than the estimates given on $\tww(G)$  in  Theorems~\ref{t:upper} and~\ref{t:lower} and those estimates remain true even if $\tww(G)$ is rescaled by a constant factor. Thus it suffices to prove these theorems with $\tww(G)$ replaced by $\stww(G)$, and that is what we will do, working from now on almost exclusively with the notion of sparse twin-width.

\section{Upper bounds}\label{s:upper}

In this section we prove \cref{t:upper}, establishing the upper bounds in our main results. 

We start by expanding on the description of our methods in the introduction and giving a   sketch of a proof of an upper bound of $n^{\frac{1}{3}+ o(1)}$ on the sparse twin-width of $n$-vertex $3$-regular graphs. While this bound is weaker than the tight bound  $n^{\frac{1}{4}+ o(1)}$, which we will end up proving, it already substantially improves on the previously known bounds and  illustrates the main steps of the approach implemented in this section.

Let $G$ be an $n$-vertex $3$-regular graph. Let $\ell = \s{\frac{1}{6}+o(1)}\log_2 n$ be an integer. For each $v \in V(G)$, let the label $f(v) \in \{0,1\}^{\ell}$ be chosen uniformly at random among all $2^{\ell}$ $\{0,1\}$-sequences  of length $\ell$. If $u_1,u_2,u_3$ are the neighbours of $v$, we define $$\phi(v) = \begin{pmatrix}
f(v) \\ f(u_1) \\ f(u_2) \\ f(u_3)
\end{pmatrix} \in \{0,1\}^{4 \times \ell }.$$
In particular, $\phi(v)$ is uniformly distributed among the $2^{4\ell}= n^{\frac{2}{3}+o(1)}$ matrices of size $4 \times \ell$ with $\{0,1\}$ values. Furthermore, if $uv \in E(G)$ then
$\phi(u)$ and $\phi(v)$ have (at least) two rows in common, although not in the same positions. Thus if $H=H(\ell)$ is the graph with all possible such  $4 \times \ell$ matrices as vertices and an edge between any two matrices sharing at least two rows in this way, then $\phi$ is a homomorphism from $G$ to $H$, and we can obtain a contraction sequence from $G$ to $H$ by arbitrarily identifying pairs of vertices of $G$ with the same image, until we obtain a subgraph of $H$. As  the dependencies between values of $\phi$ are only local, it is not hard to believe (and will be shown later in a more general and technical setting) that the number of preimages of every vertex of $H$ is typically close to the expected value $n^{\frac{1}{3}+o(1)}$ and so we can choose $\phi$ so that  the total degree of vertices identified into a single one in this process is  $n^{\frac{1}{3}+o(1)}$, i.e. our contraction sequence from $G$ to $H$ is an $n^{\frac{1}{3}+o(1)}$-contraction sequence. To show that $\stww(G) \leq n^{\frac{1}{3}+o(1)}$ it remains to extend this sequence to an $n^{\frac{1}{3}+o(1)}$-contraction sequence for $G$, i.e. to obtain an $n^{\frac{1}{3}+o(1)}$-contraction sequence for~$H$. 

To accomplish this final step note that $\Delta(H(\ell)) = n^{\frac{1}{3}+o(1)}$ and that there exists a natural homomorphism from $H(\ell)$ to $H(\ell-1)$ by deleting the final column of each matrix. As in the previous paragraph, we can convert the homomorphism  in to an  $n^{\frac{1}{3}+o(1)}$-contraction sequence from $H(\ell)$ to $H(\ell-1)$, and repeating this process eventually extend this sequence to an $n^{\frac{1}{3}+o(1)}$-contraction sequence for $H$.

This finishes the sketch. 

\vskip 10pt

We proceed to tighten and generalize the steps presented above. The next three lemmas lay down the groundwork which will allow us to obtain a random map $\phi$ from the vertices of a sparse graph $G$ to a set of matrices in a way similar to the above so that the neighbours share as much information as possible, but far away vertices share none. If a graph corresponds to a digraph with maximum outdegree one, it is possible to obtain such $\phi$ with neighbours sharing almost all the information (see \cref{l:Hom2Gamma}). In general, we obtain $\phi$ by decomposing $G$ into such digraphs.

\begin{lem}\label{l:2factors}
	If $G$ is a graph and $a,b\in \N$ are such that $a\geq b$ and $ \frac{a}{b} \geq \frac{1}{2}  \brm{mad}(G)$, then there exist spanning subgraphs $G_1,\ldots,G_a$ of $G$, each of which admits an orientation with maximum outdegree one, such that each edge of $G$ belongs to at least $b$ of them.
\end{lem}

\begin{proof}
Let $G=(V,E)$, and consider the following auxiliary bipartite graph $H$. The bipartition of $H$ consists of the two sets $V \times [a]$ and $E\times [b]$, and there exists an edge between vertices $(v,i)\in V\times [a]$ and $(e,j) \in E\times [b]$ if $v$ and $e$ are incident in $G$. We now claim that $H$ admits a matching that covers $E\times [b]$, and to prove this we use Hall's theorem. So consider any subset $X\subseteq E\times [b]$. Let $F \subseteq E$ be the projection of $X$ to the first coordinate, and let $A\subseteq V$ be the set of vertices in $G$ that touch at least one edge in $F$. Since $\mad(G)\le \frac{2a}{b}$, we have $|F|\le \frac{a}{b}\cdot |A|$. Furthermore, it follows directly by definition of $H$ that $N_H(X)=A\times [a]$. Thus, we have $|N_H(X)|=|A|\cdot a\ge b \cdot |F|\ge |X|$, where we used that $X\subseteq F \times [b]$ in the last inequality. Thus, Hall's condition is satisfied and we indeed find a matching $M$ in $H$ that covers every vertex in $E\times [b]$. For $i=1,\ldots,a$, define $E_i:=\{(u,v)|\exists t\in[b]: (u,i)\text{ is matched to }(uv,t)\text{ in }M\}$. It is readily verified that $D_i:=(V,E_i)$ is a digraph of maximum out-degree at most $1$, and its underlying graph $G_i$ is a subgraph of $G$. The fact that $M$ covers every vertex in $E\times [b]$ now directly implies that every edge $e$ of $G$ belongs to exactly $b$ of the graphs $G_1,\ldots,G_a$, as desired.
\end{proof}

In the next lemma we construct homomorphisms from graphs that admit an orientation with maximum outdegree one (which we will obtain from \cref{l:2factors}) to a certain explicit family of sparse graphs, which will serve as building blocks of the homomorphism we will use to bound $\stww(G)$.

We now define this graph family. Let $S$ be a (possibly infinite) set.
We say that two sequences $(s_1,s_2,\ldots,s_r)$ and  $(s'_1,s'_2,\ldots,s'_r)$ of the same length $r$ are \emph{close} if they contain a common subsequence of length $r-1$, that is if there exist indices $i$ and $j$ such that
$$(s_1,s_2,\ldots,s_{i-1},s_{i+1},\ldots,s_r) = (s'_1,s'_2,\ldots,s'_{j-1},s'_{j+1},\ldots,s'_r).$$

Let $\Gamma=\Gamma(S,r)$ be a graph with $V(\Gamma)=S^{r}$ such that two sequences are adjacent in $\Gamma$ if and only if they are close. Let $\Gamma_*(S,r)$ be the induced subgraph of $\Gamma$ with vertex set restricted to the set of sequences  with pairwise distinct elements. 

\begin{lem}\label{l:Hom2Gamma}
	If $G$ is a graph such that $G$ admits an orientation with maximum outdegree at most one, then for every positive integer $r$ there exists a set $S$ and a homomorphism $\phi$ from $G$ to $\Gamma_*(S,r)$ such that for every pair of vertices $u,v\in V(G)$, if the sequences $\phi(u)$ and $\phi(v)$ share a common element, then the distance between $u$ and $v$ in $G$ is at most $3r-3$.  
\end{lem}

\begin{proof} 
It suffices to prove the lemma for connected graphs $G$. Indeed, if $G$ is  a union of two vertex disjoint graphs non-null graphs $G_1$ and $G_2$, we may inductively assume that for $i=1,2$ there exist homomorhisms $\phi_i : V(G_i) \to \Gamma_*(S_i,r)$ satisfying the conditions of the lemma for $G_i$. Further, we may assume without loss of generality that $S_1 \cap S_2 = \emptyset$. If $\phi: V(G) \to   \Gamma_*(S_1 \cup S_2,r)$ is defined so that $\phi|_{V(G_i)}=\phi_i$ for $i=1,2$, then $\phi$ is as desired.

Hence we assume that $G$ is connected and that the edges of $G$ are directed so that every vertex has at most one outneighbour. Then $G$ contains a unique directed cycle $C$ or exactly one vertex with no outneighbours.

 Let $X=V(C)$ in the first case and $X=\{v\}$ in the second.

We distinguish two cases depending on $|X|$. Assume first that $|X|\geq r$.  Then for every vertex $v \in V(G)$ there exists a unique directed path of length $r$ that starts at $v$. Let $v=v_1, v_2,\ldots,v_r$ be the vertices of this path in order and set
$$ \phi(v) := (v_1,\ldots,v_r).$$
It is easy to see that $\phi$ is a homomorphism from $G$ to $\Gamma_*(V(G),r)$ and if $\phi(u)$ and $\phi(v)$ share an element then there exists a path between $u$ and $v$ in $G$ with at most $2r-2$ edges.
Thus $\phi$ satisfies the lemma in this case.

It remains to consider the case $|X|<r$. In this case, let $S = V(G) \cup \{s_1,\ldots,s_r\}$, where $\{s_1,\ldots,s_r\}$ is an arbitrary set disjoint from $V(G)$. For each $v \in V(G)$, let $v=v_1,v_2,\ldots,v_k$ be the vertices of the longest directed path starting at $v$ disjoint from $X$ (if $v\in X$, then in particular $k=0$) and of length at most $r$. Let $$ \phi(v) := (v_1,\ldots,v_k, s_1,\ldots, s_{r-k}).$$

Again it is easy to see that $\phi$ is a homomorphism from $G$ to $\Gamma_*(S,r)$. If $\phi(u)$ and $\phi(v)$ share an element in $V(G)$  then the distance between $u$ and $v$ is at most $2r-2$ as in the previous case. If $\phi(u)$ and $\phi(v)$ share an element in $\{s_1,\ldots,s_r\}$  then the distance from both $u$ and $v$ to $X$ is at most $r-1$ and as $X$ has diameter less than $r$, the distance between $u$ and $v$ is at most $3r-3$, as desired.
\end{proof}	

Next we formalize the arguments used to bound twinwidth via homomorphisms sketched at the beginning of the section, starting with the following 
simple lemma.

\begin{lem}\label{l:HomUpper}
	If $G,H$ are graphs and $\phi$ is a homomorphism from $G$ to $H$,
	then 
 \begin{equation}
	\stww(G) \leq \max\{\stww(H), m\Delta(G)\},
	\end{equation}
 where $m := \max_{v \in V(H)}|\phi^{-1}(v)|$.
\end{lem}

\begin{proof}
    It suffices to show that there exists a $m\Delta(G)$-contraction sequence from $G$ to a subgraph of $H$, as we can then extend it to a $\max\{\stww(H), m\Delta(G)\}$-contraction sequence for $G$ using the $\stww(H)$-contraction sequence for $H$. We obtain such a sequence by repeatedly identifying the vertices in $\phi^{-1}(v)$ into a single vertex for every $v$ in $V(H)$. At any step of this process each vertex of the intermediate graph $G_i$ has been obtained by identification of at most $m$ vertices of $G$ and so has degree at most $m\Delta(G)$. 
\end{proof}

Next we define a certain technical product of the graphs $\Gamma(S,r)$ defined above to which we will map our graph $G$ and bound its twin-width. 
For positive integers $q \geq 2$, $r_1,\ldots,r_a,$ and $b,$ let  $\Pi=\Pi(r_1,r_2,\ldots,r_a;b,q)$ be the graph with $$V(\Pi) = \{(v_1,v_2,\ldots,v_a) \: | \:  v_i \in [q]^{r_i}\}$$ where two vertices $\bs{v}=(v_1,v_2,\ldots,v_a)$ and $\bs{u}=(u_1,u_2,\ldots,u_a) $ are adjacent if there exists a set of $b$ components of $\bs{u}$ which are close to the corresponding components of $\bs{v}$, that is if there exists $I \subseteq [a]$ with $|I|=b$ such that for each $i \in I$ either $u_iv_i \in E(\Gamma([q],r_i))$ or $u_i=v_i$.

\begin{lem}\label{l:GammaTww}
If $a,b,q,r,r_1,\ldots,r_a\in \N$ are such that $a \geq b$, $q\geq 2$ and $r_1,\ldots,r_a \leq r$, then
	 \begin{equation}\label{e:GammaTww}
	 \stww(\Pi(r_1,r_2,\ldots,r_a;b,q)) \leq    r^{2b}  a^b q^{ar-b(r-1)+1}.
	 \end{equation}
\end{lem}

\begin{proof}
Let $\Pi=\Pi(r_1,r_2,\ldots,r_a,b)$. First, we bound $\Delta(\Pi)$.
Given  $\bs{u}=(u_1,u_2,\ldots,u_a) \in V(\Pi)$ we can enumerate its neighbours as follows.  There are $\binom{a}{b} < a^b$ ways of choosing a set $I \subseteq [a]$ with $|I|=b$ such that the components with indices in $I$ in the neighbour are close to the corresponding components of $\bs{u}$. For  each $i \in I$ there are at most $q r^2$ ways of choosing the entries of this row in the neighbour, as there are at most $r^2$ ways of choosing the shared subsequences, and at most $q$ ways of choosing the remaining entry. Finally, there are at most $q^r$ ways of choosing entries of each remaining row, yielding an upper bound $a^b \times (qr^2)^b \times (q^r)^{a-b}$ on the number of neighbours of $\bs{u}$. Thus 
\begin{equation}\label{e:DeltaGamma}
\Delta(\Pi) \leq r^{2b}  a^b q^{ar-b(r-1)}
\end{equation}
		
We now prove the lemma by induction on $\sum_{i=1}^a r_i$ using \eqref{e:DeltaGamma} and \cref{l:HomUpper}. The base case, when $r_i=1$ for every $i \in [a]$, is trivial, as $\stww(\Pi)\le \v(\Pi)=q^a<r^{2b}a^bq^{ar-b(r-1)+1}$ in this case.

For the induction step assume without loss of generality that  $r_1 \geq 2$. Let $\psi:  [q]^{r_1} \to [q]^{r_1-1}$ be a map erasing the last component, i.e. $\psi((x_1,\ldots,x_{r_1-1},x_{r_1})) = (x_1,\ldots,x_{r_1-1})$. 
	
Note that $\psi$ maps pairs of close sequences to close ones. Indeed, if deleting the $i$-{th} component  of $(x_1,\ldots,x_{r_1-1},x_{r_1})$ and $j$-th component of $(y_1,\ldots,y_{r_1-1},y_{r_1})$ yields the same sequence, then deleting $i'$-th  component  of $(x_1,\ldots,x_{r_1-1})$ and $j'$-th component of $(y_1,\ldots,y_{r_1-1})$ yields the same sequence, where $i'= \min (i, r-1)$ and  $j'=\min(j, r-1)$. Equivalently, $\psi$ is a homomorphism from $\Gamma([q],r_1)$ to $\Gamma([q],r_1-1).$
	
	Let $\Pi' = \Pi(r_1-1,r_2,\ldots,r_a;b,q)$ and let $\phi: V(\Pi) \to V(\Pi')$ be defined by $ \phi((u_1,u_2,\ldots,u_a))=(\psi(u_1),u_2,\ldots,u_a)$, i.e. $\phi$ erases the last entry of the first component of each vertex of $\Pi$. It is easily seen from the definitions that $\phi$ is a homomorphism from $\Pi$ to $\Pi'$ and $|\phi^{-1}(\bs{v})|=q$ for every $\bs v \in V(\Pi')$.
Thus by  \eqref{e:DeltaGamma}, \cref{l:HomUpper}  applied to $\phi$, and the induction hypothesis applied to $\Pi'$ we have
\begin{align*}	\stww(\Pi) \leq \max \{\stww(\Pi'), q \Delta(\Pi)\} \leq r^{2b}  a^b q^{ar-b(r-1)+1},
\end{align*}
and \eqref{e:GammaTww} holds.
\end{proof}

For the next steps in the proof we need the following rather loose variant of the Chernoff bound. 

If $X$ be a sum of independent random variables taking values in $\{0,1\}$, then 
\begin{equation}\label{e:Chernoff}\brm{P}[X  > c] \leq e^{-c/3}\end{equation}
for every $c \geq 2\brm{E}[X].$ 

We will also use the Hajnal-Szemer{\'e}di theorem on the existence of equitable colourings in graphs with a given maximum degree.

\begin{thm}[{\cite{hajnal_proof_1970}}]\label{t:HS}
	If $G$ is a graph and $k\in \N$ such that  $k\geq \Delta(G)+1$, then $V(G)$ can be partitioned into independent sets  $V_1,\ldots,V_k$ such that $|V_i| \leq \frac{\v(G)}{k}+1$.
\end{thm}

Applying \cref{t:HS} to the $\ell$-th power of $G$ we obtain the following.

\begin{cor}\label{c:HS}
	If $G$ is a graph and $\ell,k\in \N$ are such that $(\Delta(G))^{\ell}+1$, then $V(G)$ can be partitioned into sets  $V_1,\ldots,V_k$ such that $|V_i| \leq \frac{\v(G)}{k}+1$ and the distance between any pair of distinct vertices of $V_i$ in $G$ is at least $\ell+1$ in $G$.
\end{cor}

\begin{proof} Let $G_*$ be the graph with  $V(G_*)=V(G)$  such that two vertices of $G_*$ are adjacent if the  distance between them in $G$ is at most $\ell$. Then $\Delta(G_*) \leq (\Delta(G))^{\ell}$, and so the corollary follows from \cref{t:HS} applied to $G_*$.
\end{proof}	

Our next lemma gives an upper bound on the sparse twin-width of a graph in terms of several parameters, which we optimize in a subsequent lemma to get the upper bounds in our main results.

\begin{lem}\label{l:Hom2Q}
	If $G$ is a graph and $a,b,q,r\in \N$ are such that $a\geq b$, $ \frac{a}{b} \geq \frac{1}{2}  \brm{mad}(G)$, $q\geq 2$ and	
	\begin{equation}\label{e:H2Q0} \v(G) \geq  6ra\cdot q^{ar+1}(\Delta(G))^{3r},
	\end{equation}
	then
\begin{equation}\label{e:BFHCor1}
\stww(G) \leq  \max\left\{  r^{2b}  a^b q^{ar-b(r-1)+1}, \frac{3}{q^{ar}} \v(G)\Delta(G)\right\}.
\end{equation}
\end{lem}

\begin{proof}	
By \cref{l:2factors} there exist spanning subgraphs $G_1,\ldots,G_a$ of $G$, each of which admits an orientation with maximum outdegree one, such that each edge of $G$ belongs to at least $b$ of them. By \cref{l:Hom2Gamma} for each $i \in [a]$ there exists a set $S_i$ and a homomorphism $\psi_i$ from $G_i$ to $\Gamma_*(S_i,r)$ such that for every pair of vertices $u,v$, if the sequences $\psi_i(u)$ and $\psi_i(v)$ share a common element, then the distance between $u$ and $v$ in $G$ is at most $3r-3$.   

Let  $\Pi=\Pi(\underbrace{r,\ldots r}_{a \; \mathrm{times}};b,q)$. We define a random map $\phi: V(G) \to V(\Pi)$ as follows.
For each $i \in [a]$, we independently choose a random map $\pi_i: S_i \to [q]$, with the values of each individual $\pi_i$ chosen uniformly and  independently at random. Then,
let 
$$\phi(v):= (\pi_1(\psi_1(v)),\pi_2(\psi_2(v)), \ldots,\pi_a(\psi_a(v)))$$
for every $v\in V(G)$.
Note that given that $\psi_i$ is a homomorphism and by definition of $\Gamma_*(S_i,r)$, if $uv \in E(G_i)$ then $\psi_i(u)$ and $\psi_i(v)$ are close, and so $\pi_i(\psi_i(v))$ and $\pi_i(\psi_i(u))$ are also close. As each edge of $G$ belongs to at least $b$ subgraphs $G_i$ it follows that $\phi$ is a homomorphism from $G$ to $\Pi$, since adjacency in $\Pi$ is defined by having at least $b$ components which are close. 
As the components of $\psi_i(v)$ are distinct, $\pi_i(\psi_i(v))$ is uniformly distributed on $[q]^r$. As the maps $\pi_i$ are chosen independently, $\phi(v)$ is uniformly distributed on $V(\Pi)$ for each $v \in V(G)$.

Let $k=(\Delta(G))^{3r-3}+1$. By \cref{c:HS}, the vertex set of $G$ can be partitioned into sets $V_1,\ldots,V_k$ such that for every $i\in [k]$, $|V_i| \leq \frac{\v(G)}{k}+1$ and the distance between any pair of distinct vertices of $V_i$ in $G$ is at least $3r-2$ in $G$.
Thus for any pair of distinct vertices $u,v \in V_i$, the sequences $\psi_j(u)$ and $\psi_j(v)$ are disjoint for every $j\in [a]$. It follows that for each $i \in [k]$  the random variables $\{\phi(v)\}_{v \in V_i}$ are mutually independent.

We will bound $\stww(G)$ by applying \cref{l:HomUpper} to $\phi$ chosen so that $$m(\phi) := \max_{\bs u \in V(\Gamma)}|\phi^{-1}(\bs u)|$$ is as small as possible. Let 
$$c = \frac{2}{\v(\Pi)}\s{\frac{\v(G)}{k}+1}
.$$
Recalling that $\phi(v)$ is uniformly distributed on $V(\Pi)$ for every $v\in V(G)$, we have that $\bb{E}[|\phi^{-1}(\bs u) \cap V_i|] = \frac{|V_i|}{\v(\Pi)} \leq c/2$ for every $\bs u \in V(\Pi)$ and  $i \in [k]$. Thus 	$$ \bb{P}[|\phi^{-1}(\bs u) \cap V_i| > c] \leq e^{-c/3} $$ by the Chernoff bound \eqref{e:Chernoff}. It follows from the union bound that
$ \bb{P}[|\psi^{-1}(\bs u) | > kc]  \leq  ke^{-c/3} $ for every $\bs u \in V(\Pi)$, and so $ \bb{P}[m(\phi) \leq kc] \geq 1-\v(\Pi) ke^{-c/3}.$ 
 Using \eqref{e:H2Q0} we have
\begin{align*} 
    \log\s{k \v(\Pi) e^{-c/3} }
    & = \log k + ar\log q  -  \frac{2}{3\v(\Pi)}\s{\frac{\v(G)}{k}+1} \\ 
    &\leq  3r \Delta(G) + q ar  -\frac{2\v(G)}{3 \cdot q^{ar}\Delta(G)^{3r-3}}  \\
    &\leq 3r \Delta(G) + q ar  -4qar\Delta(G)  \\
    &< 0. 
\end{align*}
Thus with positive probability we have $ m(\phi) \leq  kc = \frac{2}{\v(\Pi)} \s{\v(G)+k} \leq 3\frac{\v(G)}{q^{ar}}$.   
Applying \cref{l:HomUpper} to  $\phi$ and using \cref{l:GammaTww} to bound $\stww(\Pi)$  gives $$\stww(G) \leq \max\{\stww(\Pi),  kc\Delta(G)\} \leq \max\left\{  r^{2b}  a^b q^{ar-b(r-1)+1}, \frac{3}{q^{ar}}\v(G)\Delta(G)\right\},$$
as desired.
\end{proof}

The following lemma is obtained from \cref{l:Hom2Q} by optimizing the choice of parameters.
 
 \begin{lem}\label{l:mainSparse2}
 Let $G$ be a graph and let $d := \mad(G)$ and $\ell := \log \v(G)$. If  \begin{equation}\label{e:lCond} 2 + \ell^{-1/6}  \leq d \leq \frac{\ell}{30 \log \ell}\qquad \mathrm{ and} \qquad \log(\Delta(G)) \leq \ell^{1/6},
 \end{equation}
 then
 	\begin{equation}\label{e:main2}
  \stww(G) \leq \Delta(G) \cdot (\v(G))^{\frac{d-2}{2d-2}} \cdot e^{O(d^{-1} \ell^{1/2} \log \ell +\log d)}.
  \end{equation}
 \end{lem}
 
 \begin{proof}
 Let $b :=  \left\lceil \frac{\ell^{1/2}}{d\log \ell } \right\rceil$, $a := \lceil \frac{db}{2} \rceil \leq db$, $r := \left\lfloor \frac{\ell}{2db \log \ell} \right\rfloor$. First note that
 \begin{align} 
 r & \leq \frac{\ell}{2d \log \ell\cdot \frac{\ell^{1/2}}{d\log \ell}} \leq \ell^{1/2}, \qquad \mathrm{and}  \label{e:r}\\
 	ar &\leq db \cdot  \frac{\ell}{2db\log \ell} \leq \frac{\ell}{2\log \ell} \label{e:ar}
 	\end{align}
 	
 	Let $\kappa := \frac{\ell}{(2a-b)r}$, i.e.  $\v(G)= e^{\kappa(2a-b)r}$. Finally, let  $q:=\lceil e^{\kappa} \rceil\leq 2 e^{\kappa}$. 
 	
 	Our goal is to apply \cref{l:Hom2Q}, but first let us estimate $\kappa$.
 	For the lower bound we have 	\begin{align*}
  \kappa 
  \geq \frac{\ell}{2ar} 
  \stackrel{\eqref{e:ar}}{\geq} 
  \frac{\ell}{2\frac{\ell}{2\log \ell}} 
  =  \log \ell.
 	\end{align*}
 	Thus 
  \begin{equation}\label{e:qKappa} 
  ar\log  \s{\frac{q}{e^{\kappa}}} 
  \leq ar(\log(1+e^{-\kappa})) 
  \leq are^{-\kappa} 
  \leq \frac{ar}{\ell} 
  \stackrel{\eqref{e:ar}}{\leq} 1, 	
 	\end{equation}
 	and so
  \begin{equation}\label{e:qKappa2} 
  q^{ar} 
  = \s{\frac{q}{e^{\kappa}}}^{ar} \cdot e^{\kappa a r} 
  \stackrel{\eqref{e:qKappa}}{\leq} e^{\kappa ar  +1}.
 	\end{equation}
 	Similarly for the upper bound on $\kappa$, assuming $\v(G)$, and thus $\ell$, are large enough, we have
 	\begin{align*}\frac{\ell}{\kappa} &  = (2a-b)r \geq (d-1)br \\&\geq (d-1)b \s{ \frac{\ell}{2db \log \ell} -1}  \geq  \frac{d-1}{2d} \frac{\ell}{ \log \ell} - bd \\ &\geq \frac{1}{4} \frac{\ell}{ \log \ell}- \s{\frac{\ell^{1/2}}{d\log \ell } + 1}d \geq \frac{1}{6} \frac{\ell}{ \log \ell},
 	\end{align*}
 	where the last inequality uses that $d\leq \frac{\ell}{30\log \ell}$ and that $\ell$ is large.
 	Thus we have
 	\begin{equation}\label{e:Kappalow}  \kappa \leq 6\log \ell.
 	\end{equation}

 	Next, we verify that \eqref{e:H2Q0} holds. As $a\geq \frac{db}{2}$, we have \begin{equation}\label{e:ekar}
  e^{\kappa ar} 
  = (\v(G)))^{\frac{a}{2a-b}} 
  \leq (\v(G))^{\frac{d}{2d-2}}.	\end{equation} We have that 
 	\begin{align*}
 	6 ra \cdot q^{ar+1}(\Delta(G))^{3r} 
    & \stackrel{\eqref{e:r}}{\leq} 6 ra \cdot q^{ar+1}(\Delta(G))^{3\ell^{1/2}}\\
    & \stackrel{\eqref{e:qKappa2}}{\leq} 6 \cdot ra \cdot q \cdot e^{\kappa ar  +1}(\Delta(G))^{3\ell^{1/2}}\\
    & \stackrel{\eqref{e:ekar}}{\leq} 6e\cdot ra\cdot q \cdot (\v(G))^{\frac{d}{2d-2}}\cdot (\Delta(G))^{3\ell^{1/2}}\\
    & \stackrel{\eqref{e:ar}}{\leq} 6e\cdot q \cdot \frac{\ell}{2\log \ell} \cdot (\v(G))^{\frac{d}{2d-2}}\cdot (\Delta(G))^{3\ell^{1/2}}\\
 	 &\leq 12 e \cdot  e^{\kappa} \cdot \frac{\ell}{2\log \ell}  \cdot  (\v(G))^{\frac{d}{2d-2}} \cdot (\Delta(G))^{3 \ell^{1/2}}\\ 
   & \stackrel{\eqref{e:Kappalow}}{\leq}  \ell^{7} \cdot (\v(G))^{\frac{d}{2d-2}} \cdot (\Delta(G))^{3 \ell^{1/2}},
	\end{align*}
    where the last inequality holds for $\log \ell
    \geq 6e$, and so to verify \eqref{e:H2Q0} it suffices to show that
 	\begin{equation}\label{e:H2Q+}  
 	\ell^{7}  \cdot (\Delta(G))^{3 \ell^{1/2}} \leq (\v(G))^{\frac{d-2}{2d-2}}.
 	\end{equation}
  
 	By \eqref{e:lCond} we have
 	$$ \frac{d-2}{2d-2} \geq \frac{\ell^{-1/6}}{2 + 2\ell^{-1/6}} \geq \frac{1}{4}\ell^{-1/6} $$ and therefore
  \begin{align*}7\log \ell  +  3 \ell^{1/2} \log \Delta(G) \stackrel{\eqref{e:lCond}}{\leq} 7\log \ell  +  3 \ell^{2/3} \leq \frac{1}{4}\ell^{5/6} \leq \frac{d-2}{2d-2} \ell,
 \end{align*}
 where the second inequality holds for $\ell$ large enough. Exponentiating the above we obtain \eqref{e:H2Q+}.
 
 	Thus \eqref{e:H2Q0} holds and so \eqref{e:BFHCor1}  holds by \cref{l:Hom2Q}. As
 	$\v(G) \leq q^{(2a-b)r}$, we have
   \begin{align*} \stww(G)  &
 	\leq \max\left\{  r^{2b}  a^b q^{ar-b(r-1)+1}, \frac{3}{q^{ar}}\v(G)\Delta(G)\right\} \\ &\leq 3r^{2b}  a^b q^{b+1}\Delta(G) q^{(a-b)r} \\ &=  \s{\Delta(G) \cdot (\v(G))^{\frac{d-2}{2d-2}} } \cdot \s{
	3 (r^{2b}  a^b e^{\kappa(b+1)}) \cdot \fs{q}{e^\kappa}^{(a-b)r+b+1} \cdot \frac{e^{\kappa(a-b)r}}{(\v(G))^{\frac{d-2}{2d-2}}}}.
 \end{align*}
 	Thus to verify \eqref{e:main2} it remains to show that \begin{equation}\label{e:Cor3} 3  (r^{2b}  a^b e^{\kappa(b+1)}) \cdot \fs{q}{e^\kappa}^{(a-b)r+b+1} \cdot \frac{e^{\kappa(a-b)r}}{(\v(G))^{\frac{d-2}{2d-2}}} = e^{O(d^{-1} \ell ^{1/2}\log \ell + \log d)}.
 	\end{equation}
 	We prove \eqref{e:Cor3} by showing that the logarithms of the terms on the left side are of the form $O(d^{-1} \ell ^{1/2}\log \ell + \log d)$,  starting with the last: 
 	\begin{align*}
 	\log \s{\frac{e^{\kappa(a-b)r}}{(\v(G))^{\frac{d-2}{2d-2}}}} 
  &=  \s{(a-b)r - \frac{d-2}{2d-2}(2a-b)r}\kappa \\ 
  &=\s{\frac{2}{2d-2}a-\frac{d}{2d-2}b} \kappa r\\
  &\leq  \s{\frac{2}{2d-2}\s{\frac{db}{2}+1}-\frac{d}{2d-2}b} \kappa r \\ 
  &= \frac{\kappa r}{d-1} = O\s{\frac{\log \ell \cdot \ell^{1/2}}{d}},
 	\end{align*}
 	where the last estimate uses \eqref{e:r} and \eqref{e:Kappalow}.
 	For the second term, using that $r\geq 2$ for sufficiently large $\ell$, we get $$\log\s{\fs{q}{e^\kappa}^{(a-b)r+b+1}} \leq ar\log  \s{\frac{q}{e^{\kappa}}}  \stackrel{\eqref{e:qKappa}}{\leq} 1.$$ 
 	Finally, for the first term,
 	\begin{equation}\label{e:last}
 	\log \s{r^{2b}  a^b e^{\kappa(b+1)} } =  b \cdot O(  \log r + \log d + \log b + \kappa) = b \cdot O(\log d + \log \ell).
 	\end{equation}
 	If $d \leq \frac{\ell^{1/2}}{\log \ell }$ then $\log d = O( \log \ell)$ and $b = O (  \frac{\ell^{1/2}}{d\log \ell })$ and so the expression in  \eqref{e:last} has the form $O ( d^{-1}\ell^{1/2})$ as desired. Otherwise, $b=1$ and $\log \ell = O(\log d)$ and so the expression in \eqref{e:last} has the form $O(\log d)$. This yields the desired bound for the first term in \eqref{e:Cor3}, and so \eqref{e:main2} holds.
  \end{proof}

\cref{t:upper}, which we restate for convenience,  readily follows from \cref{l:mainSparse2}.\upper*

\begin{proof}
	Let $G$ be an $n$-vertex graph, $\ell:=\log n$ and $d := \mad(G)$.

	If $d = \Omega(\frac{\ell}{\log \log \ell})$ then by \eqref{e:mUpper} we have
	$$\tww(G) = O(\sqrt{dn}) = n^{\frac{d-2}{2d-2}} O(\sqrt{d}\cdot n^{\frac{1}{2d-2}})  = n^{\frac{d-2}{2d-2}} e^{O\left(\frac{\ell}{d} + \log d\right)} =  n^{\frac{d-2}{2d-2}} e^{O( \log d)}.$$
	Thus \eqref{e:mainUpper} holds, and so we may assume that $d \leq \frac{\ell}{30 \log \ell}$. As it suffices to establish \eqref{e:mainUpper} when $n$ is sufficiently large (and sufficiently large as a function of $\eps$ in (ii)),  we assume from now on that $\ell$ is large enough to satisfy subsequent inequalities. In particular, we can  assume $\ell^{-1/6} \leq \eps/2$ in (ii), and so that $d \geq 2 + \ell^{-1/6}$.
	
	Given  $2 + \ell^{-1/6}  \leq d \leq \frac{\ell}{30 \log \ell}$, \eqref{e:mainUpper} follows from \cref{l:mainSparse2}, as long as \begin{equation}\label{e:logDelta} \log(\Delta(G)) \leq \ell^{1/6} \qquad \mathrm{and}  \qquad \log \Delta(G) = O(d^{-1}\ell^{1/2}\log \ell + \log d).\end{equation}  Clearly, \eqref{e:logDelta} holds in (i) as $\Delta(G) = d$ in this case.
	 
	 It remains to verify that  \eqref{e:logDelta} holds with high probability if $G$ is as in (ii). It is well-known and is an easy consequence of Chernoff type bounds that $\Delta(G) = O (\ell \frac{m}{n}) = O(\ell d)$ with high probability, and so  $\log \Delta(G) = O(\log \ell)$ with high probability. Thus \eqref{e:logDelta}  holds with high probability as  $d^{-1}\ell^{1/2}\log \ell + \log d \geq \frac{1}{2}\log \ell$ (which can be easily seen by splitting into cases $d\geq l^{1/2}$ and $d\leq l^{1/2}$).
\end{proof}
 
\section{Lower bounds}\label{s:lower}
In this section we prove \cref{t:lower}, establishing the lower bounds in our main results. 

We start by showing that for every graph $G$ with  $\stww(G) \leq w$ we can divide the vertices of $G$ into roughly equal parts, close to any prescribed size, so that identifying all vertices in each part yields a graph with average degree close to $w$. As mentioned in the introduction, this fact will be the basis of our counting argument.
  
If $G$ is a graph and $\mathcal P$ is a partition of $V(G)$, we define the \emph{quotient graph} $G/\mathcal P$ as the graph on vertex set $\mathcal P$ in which distinct $P_1,P_2\in \mathcal P$ are adjacent if and only if there exists at least one edge in $G$ between a vertex in $P_1$ and a vertex in $P_2$. If $X\subseteq V(G)$, we denote by $\mathcal P-X$ the partition of $G\setminus X$ obtained by restricting each part of $\mathcal P$ to $V(G)\setminus X$. We say $G$ is $w$-degenerate if every subgraph of $G$ contains a vertex of degree at most $w$.

\begin{lem}\label{lem:Kpartition}
	If $G$ is a graph and $K<\v(G)$ is an integer, then there exists a  partition $\mathcal P$ of $V(G)$ such that $|\mathcal P| \leq K$, $|P| \leq \frac{2\v(G)}{K}$ for every $P \in \mc{P}$ and  $G/\mathcal P$ is $\stww(G)$-degenerate.
\end{lem}

\begin{proof}
	Let $n:=\v(G)$ and $w:=\stww(G)$. Let $\mathcal P_1,\dots,\mathcal P_n$ be the sequence of partitions of $V(G)$ such that
	\begin{enumerate}
		\item  $|\mc{P}_i|=n-i+1$ for every $i\in [n]$,
		\item  $\mc{P}_{i+1}$ is a coarsening of $\mc{P}_i$ for every $i\in [n-1]$, and
		\item $G=G/\mc{P}_1,G/\mc{P}_2,\ldots, G/\mc{P}_n$ is a $w$-contraction sequence for $G$,
	\end{enumerate}
    which exists by definition of $\stww(G)$ and contraction sequences.

	We define a partition $\mathcal P=(P_1,\dots,P_k)$ for some $k \leq K$ recursively as follows. Suppose we have already defined $P_1,\dots,P_{r-1}$. Let $X_{r-1}=P_1\cup\dots\cup P_{r-1}$. If $V(G)\setminus X_{r-1}=\emptyset$, we let $k=r-1$. Otherwise, we wish to define $P_r$.
	
	Let $i_r$  be the minimum index $i\leq n$ such that $\mathcal P_{i} -X_{r-1}$ contains a part of size at least $\frac{\v(G)}{K}$ and let $P_r$ be this part. If no such index $i$ exist, let $i_r=n$ and set $P_r=V(G)\setminus X_{r-1}$. Note that in the latter case, $P_r$ is the sole part of $\mathcal P_n-X_{r-1}$, so we have $|P_r|<\frac{\v(G)}{K}$.
	
	We now show that $\mc{P}=(P_1,\dots,P_k)$ satisfies the requirements of the lemma.
	Clearly $\mathcal P$ is a partition of $G$. We have $|P_r| \geq  \frac{\v(G)}{K}$ for $r \in [k-1]$ and $P_k \neq \emptyset$. Thus $\v(G) = \sum_{r=1}^{k}|P_r| > (k-1) \frac{\v(G)}{K}$, thus $|\mathcal P| =k \leq K$ as desired.
	
	If $i_r=1$, $|P_r|=1<\frac{\v(G)}{K}$. Otherwise, $i_r>1$. If $P_r$ was chosen in the case having size at least $\frac{\v(G)}{K}$, it is a union of at most two parts of $\mathcal P_{i_r-1} -X_{r-1}$, and each part of $\mathcal P_{i_r-1} -X_{r-1}$ has size less than $\frac{\v(G)}{K}$ by the choice of $i_r$, and so it follows that $|P_r| \leq \frac{2\v(G)}{K}$. Otherwise, we have noted above that $|P_r|<\frac{\v(G)}{K}$. Hence, $|P_r| \leq \frac{2\v(G)}{K}$ for every $r\in [k]$.
	
	We now claim that $G/\mathcal P$ is $w$-degenerate. It suffices to show that for every $r\in [k-1]$, there are at most $w$ parts among  $P_{r+1},\dots,P_{k}$ with neighbours in $P_r$. Indeed, $P_r=P'_r-X_{r-1}$ for some $P'_r \in \mc{P}_{i_r}$. As $G / \mc{P}_{i_r}$ has maximum degree at most $w$, there are parts $\hat P_1,\hat P_2,\ldots,\hat P_s\in \mc{P}_{i_r}$ for some $s \leq w$ such that every neighbour of $P_r$ in $G-X_{r-1}$ belongs to one of them. As $i_r \leq i_{r+1} \leq \ldots \leq i_k$, each part $\hat P_j$ intersects at most one of $P_{r+1},\dots,P_{k}$, implying the claim.

    This concludes the proof of the lemma.
\end{proof}

In estimates below we frequently use the following standard bounds on binomial coefficients: $$\binom{n}{m} \leq \fs{en}{m}^m$$
for all integer $n \geq m > 0$.

\begin{lem}\label{l:Lower} There exists $C_0 > 0$ satisfying the following. For every $0 < \eps \leq 1/3$, there exists $n_0\in \N$ such that for $n,m\in \N$ such that $n \geq n_0$ and $ (1 + \eps)n \leq m \leq  \frac{1}{10}n{\log n}$ there are at most $ \s{C_0\eps \frac{n^2}{m}}^{m} $
graphs $G$ with $V(G)=[n]$, $\e(G)=m$ and $$\stww(G) \leq \eps d \fs{n}{\log n}^{\frac{d-2}{2d-2}},$$
where $d= \frac{2m}{n}.$
\end{lem}
 
\begin{proof}
Set $C_0:=24000$. Let $\varepsilon$ be given.

Choose $n_0$ large enough such that $\frac{\varepsilon}{2}\left(\frac{n_0}{\log n_0}\right)^{\frac{\varepsilon}{2+4\varepsilon}}\geq 1$, and further such that $\frac{2}{d}+\frac{4d}{\log n} \leq 1$ for every $n,d$ such that $n\ge n_0$ and $2+2\eps \leq d \leq \frac{1}{5}\log n$.

Let $n,m,d$ be as in the statement.

Let $w := \left\lfloor  \eps d \fs{ n}{\log n}^{\frac{d-2}{2d-2}} \right\rfloor$ and $K :=\left\lfloor \frac{d^2n}{w\log n}\right\rfloor$. 

By our choice of $n_0$, we can ensure that $w\geq \frac{\eps}{2} d \fs{ n}{\log n}^{\frac{d-2}{2d-2}} \ge d$ and thus $K\le \frac{dn}{\log n}\le \frac{n}{5}<n$. Furthermore, as $w<d\frac{n}{\log n}$ we have that $\frac{d^2n}{w\log n}\geq d> 2$ and so $K> \frac{d^2n}{2w\log n}$.
	
By \cref{lem:Kpartition}, if $G$ is a graph with $V(G)=[n]$ and $\stww(G) \leq w$, then there exists a partition $\mc{P}=(P_1,\ldots,P_K)$ of $[n]$ such that $|P_i| \leq \frac{2n}{K}$ for every $i \in K$ and $G/\mathcal P$ is $w$-degenerate. We wish to compute the number of graphs $G$ for which this is possible.
 
 For a given partition $\mathcal P=(P_1,\ldots,P_K)$ of $[n]$ and $F \subseteq \binom{[K]}{2}$, let $E(\mc{P},F)$ be the set of all $\{u,v\} \in [n]^2$ such that either $u,v \in P_i$ for some $i \in [K],$ or $u \in P_i$ and $v \in P_j$ for some $\{i,j\} \in F$. In other words, $E(\mc{P},F)$ is the set of possible edges for a graph $G$ on $[n]$ if $E(G/\mathcal P)\subseteq F$ (writing $V(G/\mathcal P)=[K]$ for simplicity).

 	Let $N$ be the number of graphs $G$ with $V(G)=[n]$, $\e(G)=m$ and $\stww(G) \leq w$. Every such graph $G$ is determined by the choice of a partition $\mc{P}$ of $[n]$ with $|\mc{P}|=K$, a set $F \subseteq \binom{[K]}{2}$ with $|F|=wK$ (since it follows from degeneracy that $\e(G/\mathcal P) \leq wK$) and a choice of $E(G) \subseteq E(\mc{P},F)$ with $|E(G)|=m$.
	 
	There are at most $K^n$ choices of $\mc{P}$. The number of choices of $F$ is upper bounded by  
 \begin{align*}
 \binom{\binom{K}{2}}{wK} 
 &\leq \fs{eK}{2w}^{wK} 
 \leq \fs{ed^2n}   {2w^2\log n}^{\frac{d^2 n}{\log n}}  
 \leq \fs{ed^2n}   { \frac{\eps^2}{4} d^2 \fs{ n}{\log n}^{\frac{2d-4}{2d-2}} \log n}^{\frac{d^2 n}{\log n}} \\ 
 &\leq \fs{4e}{\eps^2}^{\frac{d^2 n}{\log n}} n^{\frac{d^2 n}{(d-1)\log n}} 
 =\fs{4e}{\eps^2}^{\frac{2dm}{\log n}} e^{\frac{2dm}{d-1}} 
 \leq e^{4m}\cdot \fs{4e}{\eps^2}^{\frac{2d m}{\log n}}
 \end{align*}
	Finally, using that 
 $$|E(\mc{P},F)| = \sum_{\{i,j\}\in F}|P_i||P_j| +  \sum_{i \in [K]}\binom{|P_i|}{2} \leq (|F|+K)\fs{2n}{K}^2= K(w+1)\fs{2n}{K}^2\leq \frac{8wn^2}{K},$$ the number of choices of  $E(G) \subseteq E(\mc{P},F)$ with $|E(G)|=m$ is upper bounded by 
 $$ \binom {|E(\mc{P},F)|}{m} \leq \fs{e\frac{8wn^2}{K}}{m}^{m} = \fs{16ewn}{dK}^{m}.$$
Combining these estimates we obtain
\allowdisplaybreaks
	\begin{align*} 
 N^{\frac{1}{m}} 
 &\leq K^{2/d} \cdot e^{4} \fs{4e}{\eps^2}^{\frac{2d}{\log n}} \cdot \frac{16ewn}{dK}\\
 &= 16 e^5 \cdot (4e)^\frac{2d}{\log n}\cdot\eps^{-\frac{4d}{\log n}} \cdot K^{2/d-1}\cdot w \cdot \frac{n}{d}\\
  &\leq 8 e^5 \cdot (4e)^\frac{2\cdot \frac{1}{10}\log n}{\log n}\cdot\eps^{-\frac{4d}{\log n}} \left(\frac{2w\log n}{2 d^2 n}\right)^{\frac{d-2}{d}}\cdot w \cdot \frac{n^2}{m}\\
  &\leq 16 e^5 \cdot (4e)^{\frac{1}{5}}\cdot\eps^{-\frac{4d}{\log n}} \cdot \left(\frac{\log n}{ n}\right)^{\frac{d-2}{d}}\cdot d^{-\frac{2d-4}{d}}\cdot w^{\frac{2d-2}{d}} \cdot \frac{n^2}{m}\\
  &\leq 8000\cdot\eps^{-\frac{4d}{\log n}} \cdot \left(\frac{\log n}{ n}\right)^{\frac{d-2}{d}}\cdot d^{-\frac{2d-4}{d}}\cdot \left(\eps d \fs{ n}{\log n}^{\frac{d-2}{2d-2}}\right)^{\frac{2d-2}{d}} \cdot \frac{n^2}{m}\\
  &= 8000\cdot\eps^{2-\left(\frac{2}{d}+\frac{4d}{\log n}\right)} \cdot d^{\frac{2}{d}}\cdot \frac{n^2}{m}\\
  &\leq 24000\cdot\varepsilon \cdot\frac{n^2}{m}=C_0\cdot\varepsilon\cdot\frac{n^2}{m},
	\end{align*}
the last inequality following from our choice of $n_0$.

This concludes the proof of the lemma.
\end{proof}	 

\cref{l:Lower} immediately implies the desired lower bounds on the sparse twin-width of $d$-regular graphs, when combined with the following rather loose estimate of the number of $n$-vertex $d$-regular graphs, which is implied for example by a much more precise and general result of McKay and Wormald~\cite{mckay_asymptotic_1991}.

\begin{thm}[see ~\cite{mckay_asymptotic_1991}]\label{t:Dnumber} 
	There exists $C_1 > 0$ such that for all $n,d$ such that $d \leq \log n$ and  $dn$ is even there are at least $(C_1\frac{n}{d})^{\frac{dn}{2}}$  $d$-regular graphs $G$ with $V(G)=[n]$.
\end{thm}

We now derive Theorem~\ref{t:lower} (i), which we restate here as a corollary, from \cref{l:Lower} and \cref{t:Dnumber}.

\begin{cor}[Theorem~\ref{t:lower} (i)]
 There exists $\eps > 0$ such that if $G$ is a uniformly random $n$-vertex, $d$-regular graph with $2 < d < \frac{1}{5}\log n$, then 
 $$\stww(G) \geq \eps\frac{d}{\sqrt{\log n}}n^{\frac{d-2}{2d-2}}$$
 with high probability.
\end{cor} 	

\begin{proof}
	Let $\eps := \frac{C_1}{3 C_0}$, where $C_0$ and $C_1$ are the constants satisfying the conditions of \cref{l:Lower} and \cref{t:Dnumber}, respectively.  We may suppose that $n\geq n_0$, where $n_0$ is obtained from \cref{l:Lower}, depending on $\varepsilon$. Let $w := 	\eps\frac{d}{\sqrt{\log n}}n^{\frac{d-2}{2d-2} }. $ Then by \cref{l:Lower} there are at most $$\s{C_0\eps \frac{2n}{d}}^{\frac{dn}{2}}$$ $d$-regular graphs on $n$-vertices with $\stww(G) < w$, and by \cref{t:Dnumber} it follows that the probability that $\stww(G) < w$ is at most
	$$ \s{C_0\eps \frac{2n}{d}}^{\frac{dn}{2}} / \s{C_1\frac{n}{d}}^{\frac{dn}{2}} =  \fs{2C_0\eps}{C_1}^{\frac{dn}{2}} = \fs{2}{3}^{\frac{dn}{2}}\leq \left(\frac{2}{3}\right)^n\xrightarrow{n\rightarrow\infty} 0.$$
	Thus $\stww(G)\geq w$ w.h.p., as desired.
\end{proof}	

It remains to prove Theorem~\ref{t:lower} (ii). This requires more preparation, as in this setting we will not be applying \cref{l:Lower} to $G$ itself, but to the subgraph of $G$ attaining the maximum average degree. First, we will need to ensure that this subgraph is not too small.

Let $\mc{G}_{n,m}$ denote the set of all graphs $G$ with $V(G)=[n]$ and $\e(G)=m$.
 We say that a graph $G$ with $n$ vertices and $m$ edges is \emph{$\alpha$-balanced} if $\e(G[X]) \leq \frac{m}{n}|X|$ for every $X \subseteq [n]$ with $|X| \leq \alpha (n-1)$.

\begin{lem}\label{l:balanced}
	If $0<\eps<0.1$, $\alpha = \frac{1}{4}e^{-\frac{4}{\eps}}$ and $n,m\in \N$ are such that $(1+\varepsilon)n\leq m\leq \frac{n^2}{8}$, almost all graphs in $\mc{G}_{n,m}$ are $\alpha$-balanced.
\end{lem}

\begin{proof} 
    Given the upper bound on $m$, we may assume that $n$ is large enough such that $\binom{n}{2}-m\geq \frac{\binom{n}{2}}{e}$ (and $n\geq e^2+1$).

	We start with an auxiliary computation. Let $d := \frac{m}{n}\geq 1+\varepsilon$ and
	let $$f(x) := e^{(2d+1)x} \fs{x-1}{n-1}^{(d-1)x}. $$
	Note that for $2 \leq x \leq \alpha(n-1)$, we have 
 \begin{align*}
     \frac{f(x+1)}{f(x)} 
     &= e^{2d+1}\fs{x}{x-1}^{(d-1)x}\fs{x}{n-1}^{d-1} 
     \leq e^{2d+1}4^{d-1}\alpha^{d-1}=e^3 (4e^2\alpha)^{d-1}\\
     &=e^3(e^{-\frac{4}{\varepsilon}+2})^{d-1}
     \leq e^3(e^{-\frac{4}{\varepsilon}+2})^{\varepsilon}
     =e^{2\varepsilon-1}\leq \frac{1}{2},
 \end{align*}
where for the first inequality sign, we used that $(\frac{x}{x-1})^x$ is monotonically decreasing for $x>1$ and that $x\ge 2$. 
	It follows that $f(x) \leq f(2)2^{2-x}$ for all integers $2 \leq x \leq \alpha(n-1)$. 
	
	Moving on to the main proof, for $X \subseteq [n]$ let $\mc{B}(X)$ denote the set of all graphs  $G \in \mc{G}_{n,m}$ such that $\e(G[X]) > d|X|$. Let $x := |X|$ and $k := \lceil dx \rceil$. Then 
 \begin{equation}
 |\mc{B}(X)| \leq \binom{\binom{x}{2}}{k} \binom{\binom{n}{2}}{m-k},
 \end{equation}
	as there are at most $\binom{\binom{x}{2}}{k}$ ways of choosing $k$ edges of $G[X]$ and at most $\binom{\binom{n}{2}}{m-k}$ choices for the remaining edges.
	Thus
 \begin{align*} 
 \frac{|\mc{B}(X)|}{|\mc{G}_{n,m}|} 
 &\leq  \frac{\binom{\binom{x}{2}}{k} \binom{\binom{n}{2}}{m-k}}{\binom{\binom{n}{2}}{m}}
=  \binom{\binom{x}{2}}{k}\frac{ \frac{\binom{n}{2}!}{(m-k)!\left(\binom{n}{2}-(m-k)\right)!}}{\frac{\binom{n}{2}!}{m!\left(\binom{n}{2}-m\right)!}}\\
&=  \binom{\binom{x}{2}}{k}
\frac{m(m-1)\dots(m-k+1)}{\left(\binom{n}{2}-m+k\right)\left(\binom{n}{2}-m+k-1\right)\dots\left(\binom{n}{2}-m+1\right)}\\
&\leq \left(\frac{e\binom{x}{2}}{k}\right)^k
\frac{m^k}{\left(\binom{n}{2}-m\right)^k}
\leq \left(\frac{e\binom{x}{2}}{dx}
\cdot \frac{dn}{\binom{n}{2}/e}\right)^k\\
&= \left(\frac{e^2 (x-1)}{n-1}\right)^k
\leq \left(\frac{e^2 (x-1)}{n-1}\right)^{dx}.
	\end{align*} 

We upper bound the proportion of graphs in $\mc{G}_{n,m}$ which are not  $\alpha$-balanced by 
	\begin{align*} 
 \sum_{X \subseteq [n], 0 < |X| < \alpha n}
 &\frac{|\mc{B}(X)|}{|\mc{G}_{n,dn}|} 
 \leq \sum_{1 \leq  x < \alpha n} \binom{n}{x}\fs{e^2(x-1)}{(n-1)}^{dx} 
 \leq \sum_{2 \leq  x < \alpha n}\fs{en}{x}^{x}\fs{e^2(x-1)}{(n-1)}^{dx} \\ 
 &\leq  \sum_{2 \leq  x < \alpha n} e^{(2d+1)x} \fs{x-1}{n-1}^{(d-1)x}  
 = \sum_{2 \leq  x < \alpha n} f(x) 
 \leq f(2) \sum_{2 \leq  x < \alpha n} 2^{2-x} \\ 
 &\leq 2 f(2) = \frac{2e^{2(2d+1)}}{(n-1)^{2(d-1)}}=2e^6\left(\frac{e^2}{(n-1)}\right)^{2(d-1)}
 \leq 2e^6\left(\frac{e^2}{(n-1)}\right)^{2\varepsilon}.	
	\end{align*} 
	As the last term tends to $0$ as $n \to \infty$ the lemma follows.
\end{proof}

The next lemma states that there is no large part of $G$ which is too dense, which will allow us to apply \cref{l:Lower}. We will need the following Chernoff bound. If $X$ is a sum of independent random variables taking values in $\{0,1\}$, then \begin{equation}\label{eq:chernoff}
    \brm P\left[X\geq (1+\varepsilon)\brm E[X]\right]\leq e^{-\frac{\varepsilon^2}{2+\varepsilon}\brm E[X]}
\end{equation}
for every $\varepsilon>0$.

\begin{lem}\label{lem:xlogx}
    Let $\alpha>0$ and let $n,m\in \N$ be such that $n \leq  m \leq \frac{1}{11}n\log n$. If $G$ is a uniformly random graph in $\mathcal G_{n,m}$, then with high probability for every set $X\subseteq [n]$ with $|X|\geq \alpha n+1$ we have
    \begin{equation}\label{eq:xlogx}
        \e(G[X])\leq \frac{1}{10}|X|\log|X|.
    \end{equation}
\end{lem}

\begin{proof}
    We may first suppose that $n$ is large enough such that $\alpha n\geq 3$.

    Say a graph $G$ with vertex set $[n]$ has property (\textasteriskcentered) if \eqref{eq:xlogx} holds for every $X\subseteq [n]$ with $|X|\geq \alpha n+1$.
    
    We first claim that if $G$ is a random graph $G(n,p)$, with $p=\frac{m}{\binom{n}{2}}$, then (\textasteriskcentered) holds.

    Let $X\subseteq [n]$ such that $|X|\geq \max(\alpha n+1,3)$ be fixed and $x=|X|$. Then, $\e(G[X])$ is a sum of independent binomial variables (one for each possible edge). We have  $\brm E[\e(G[X])]=p\cdot\binom{x}{2}=\frac{m\cdot\binom{x}{2}}{\binom{n}{2}}= \frac{mx(x-1)}{n(n-1)}$.

    Note that the function $\frac{\log x}{x}$ is monotone decreasing for $x\geq e$. In particular, as $x\leq n$, $\frac{\log x}{x}\geq \frac{\log n}{n}$. Suppose that $n f_1(n)\leq m\leq n f_2(n)$; we will have to consider two different regimes. However we will always have $f_2(n)\leq \frac{1}{11}\log n$, and so $\frac{\log n}{10f_2(n)}>1$. We then have that
    \begin{align*}
        \frac{1}{10}x\log x
        &= \frac{1}{10}\cdot \frac{\log x}{x}x^2
        \geq\frac{1}{10}\cdot \frac{\log n}{n}x^2\cdot \frac{m}{n f_2(n)}
        =\frac{\log n}{10 f_2(n)}\cdot \frac{mx^2}{n^2}\\
        &\geq\frac{\log n}{10 f_2(n)}\cdot \frac{mx(x-1)}{n(n-1)}
        =\left(1+\left(\frac{\log n}{10 f_2(n)}-1\right)\right)\brm E[\e(G[X])].
    \end{align*}
    
    Hence, by the Chernoff bound above (\cref{eq:chernoff}),
    \begin{align*}
        \brm P\left[\e(G[X])\geq \frac{1}{10}x\log x\right]
        &\leq \brm P\left[\e(G[X])\geq \left(1+\left(\frac{\log n}{10 f_2(n)}-1\right)\right)\brm E[\e(G[X])]\right]\\
        &\leq \exp\left(-\frac{\left(\frac{\log n}{10 f_2(n)}-1\right)^2}{2+\left(\frac{\log n}{10 f_2(n)}-1\right)}\cdot\frac{mx(x-1)}{n(n-1)}\right)\\
        &\leq \exp\left(-\frac{\left(\frac{\log n}{10 f_2(n)}-1\right)^2}{1+\frac{\log n}{10 f_2(n)}}\cdot\frac{n f_1(n)\cdot (\alpha n)^2}{n^2}\right)\\
        &< \exp\left(-\alpha^2\frac{\left(\frac{\log n}{10 f_2(n)}-1\right)^2}{\frac{\log n}{5 f_2(n)}}f_1(n) \cdot n \right).
    \end{align*}

If $n\leq m\leq n\sqrt{\log n}$, i.e. $f_1(n)=1$ and $f_2(n)=\sqrt {\log n}$, then this value becomes
$$\exp\left(-\alpha^2\frac{\left(\frac{\log n}{10 \sqrt{\log n}}-1\right)^2}{\frac{\log n}{10 \sqrt{\log n}}}n\right)
=\exp\left(-\frac{\alpha^2}{20}\left(\sqrt{\log n}-20+\frac{100}{\sqrt{\log n}}\right)n\right)\ll 2^{-n}.$$

Otherwise, if $n\sqrt{\log n}\leq m\leq \frac{1}{11}n\log n$, i.e. $f_1(n)=\sqrt{\log n}$ and $f_2(n)=\frac{1}{11}\log n$, then this value becomes
$$\exp\left(-\alpha^2\frac{\left(\frac{\log n}{10 \cdot \frac{1}{11}\log n}-1\right)^2}{\frac{\log n}{5 \cdot \frac{1}{11}\log n}}\sqrt{\log n}\cdot n \right)
=\exp\left(-\frac{\alpha^2}{220}\sqrt{\log n} \cdot n\right)\ll 2^{-n}.$$

Given that are at most $2^n$ possible choices of $X$, in all cases the probability that there is some $X$ with $|X|\geq \alpha n+1$ for which \eqref{eq:xlogx} does not hold converges to $0$ as $n\rightarrow\infty$. This completes the proof of the claim.

We now show the statement when $G$ is uniformly chosen in $\mathcal G_{n,m}$ with a coupling argument.

Let $G(n,p,m)$ denote the random graph distribution defined by conditioning $G(n,p)$ on having at least $m$ edges. The number of edges of $G(n,p)$ follows a binomial distribution with parameters $B(\binom{n}{2},p)$. It is well-known (see \cite{lord_binomial_2010}, for instance) that when the mean of a binomial distribution is an integer, it is also equal to the median. In particular, for our previous choice $p=\frac{m}{\binom{n}{2}}$,
$$\brm P(\e(G(n,p)\geq m)\geq \frac{1}{2}.$$ It thus also holds that if $G$ is a random graph $G(n,p,m)$, with high probability (\textasteriskcentered) holds.

Let $G$ and $G'$ be random graphs obtained using the following process. $G'$ will have distribution $G(n,p,m)$, and $G$ will be chosen as a uniformly random subgraph of $G'$ with exactly $m$ edges. It is immediate by symmetry that $G$ is a uniformly random graph in $\mathcal G_{n,m}$. Furthermore, $E(G)\subseteq E(G')$ and so $\e(G[X])\leq \e(G'[X])$ for any $X\subseteq [n]$. As (\textasteriskcentered) holds with high probability for $G'$, it thus follows that with high probability (\textasteriskcentered) holds for $G$. This concludes the proof of the statement.
\end{proof}

We are now ready to start the proof of \cref{t:lower} (ii), which we restate here.

\begin{thm}[\cref{t:lower} (ii)]\label{t:lower2} 
	For every $\eps > 0$ there exists $\delta > 0$ such that for all $n,m\in \N$ such that $(1+\eps)n \leq  m \leq \frac{1}{11}n\log n$ the following holds. If $G$ is a uniformly random graph in $\mathcal G_{n,m}$, then with high probability  we have 	
	\begin{equation}\label{e:lower2}	\tww(G) \geq 
	\delta \frac{d}{\sqrt{\log n}}n^{\frac{d-2}{2d-2} }, 		
	\end{equation}
	where $d = \mad(G)$.
\end{thm}

\begin{proof}
W.l.o.g.~suppose $\varepsilon<0.1$, let $\alpha = \frac{1}{4}e^{-\frac{4}{\eps}}$ and let $C_0$ be as in \cref{l:Lower}. We show that $\delta = \frac{\alpha^{3/2} 4^{-2/\alpha}}{12C_0}$ satisfies the theorem.  We may suppose that $n$ is sufficiently large such that $\alpha(n-1) \geq \frac{\alpha}{2}n$, $\binom{n}{2}-m\geq n^2/3$, which is possible given the upper bound on $m$, and $\alpha(n-1)\geq n_0$, where $n_0$ is obtained from \cref{l:Lower}.

Let $S(X,m')$ be the set of all graphs in $\mathcal G_{n,m}$ such that $\e(G[X])=m'$ and $X$ is a maximal subset of $[n]$ satisfying $\mad(G)=\frac{2\e(G[X])}{|X|}$, i.e. such that $\mad(G)$ is the average degree of $G[X]$.  Note that $\mad(G)\geq \frac{2e(G)}{\v(G)}=\frac{2m}{n}$.

By \cref{l:balanced}, with high probability every $X\subseteq [n]$ with $|X|\leq \alpha (n-1)$ is such that $\e(G[X])\leq \frac{m}{n}|X|$. In particular, for such an $X$, $\frac{2\e(G[X])}{|X|}\leq \frac{2m}{n}\leq \mad(G)$. Hence, no such $X$ has average degree greater than that of the entire graph $G$. Hence, with high probability $G$ has a set $X\subseteq [n]$ which attains the maximum average degree and has size $|X|\geq \alpha(n-1)$.

On the other hand, \cref{lem:xlogx} states that with high probability every set $X$ with $|X|\geq \alpha(n-1)$ is such that $\e(G[X])\leq \frac{1}{10}|X|\log|X|$.

Hence, with high probability $G$ belongs to some set $S(X,m')$ where $|X| \geq \alpha(n-1)$ and $m'\leq \frac{1}{10}|X|\log|X|$. Furthermore, as $X$ attains the maximum average degree, then $\frac{2m'}{|X|}\geq \frac{2m}{n}$ and so $m'\geq \frac{m}{n}|X|\geq (1+\varepsilon)|X|$.

Let $w := \delta \frac{d}{\sqrt{\log n}}n^{\frac{d-2}{2d-2} } $. Let $N(X,m')$ be the  number of graphs  $G \in S(X,m')$ with 	$\tww(G) \leq w$.
As there are at most $2^n$ choices of $X$ and $o(2^n)$ choices of $m'$, to prove the theorem it suffices to show that    \begin{equation}\label{e:nxm} 4^n N(X,m') \leq |\mathcal G_{n,m}| \end{equation} for any $X$ and $m'$ as above.

Let $x:=|X|$ and $\eps':=\frac{2\delta}{\alpha^{1/2}}$. Note that $$\varepsilon'=\frac{\alpha\cdot 4^{-2/\alpha}}{12C_0}<\alpha=\frac{1}{4}e^{-\frac{4}{\varepsilon}}<\varepsilon.$$

As $n\geq x \geq \alpha(n-1) \geq \frac{\alpha}{2}n$, we have $$\eps' d\left(\frac{x}{\log x}\right)^{\frac{d-2}{2d-2}} \geq \eps' \sqrt{\frac{\alpha}{2}} \cdot\frac{d}{\sqrt{\log n}} n^{\frac{d-2}{2d-2}} \geq w.$$ Hence, as $\stww(G) \leq w$ we have $\stww(G[X]) \leq w \leq  \eps' \left(\frac{x}{\log x}\right)^{\frac{d-2}{2d-2}}$. Given that $x\geq \alpha (n-1)\geq n_0$, $(1+\varepsilon')x\leq (1+\varepsilon)x\leq m'\leq \frac{1}{10}x\log x$, by \cref{l:Lower} there are at most $\s{C_0 \eps' \frac{x^2}{m'}}^{m'}$ possible choices of for the subgraph $G[X]$ such that $G \in S(X,m')$ and $\tww(G) \leq w$. There are also at most $\binom{\binom{n}{2}-\binom{x}{2}}{m-m'}\leq \binom{\binom{n}{2}}{m-m'}$ choices for the remaining edges of $G$.
It follows that $$ N(X,m') \leq \s{C_0 \eps' \frac{x^2}{m'}}^{m'}\binom{\binom{n}{2}-\binom{x}{2}}{m-m'}.$$ Using $m \leq \frac{n}{x}m' \leq \frac{2}{\alpha}m'$, we obtain 
\begin{align*} 
 &\left(\frac{4^nN(X,m')}{|G_{n,m}|}\right)^{1/m'} 
 \leq 4^{n/m'} C_0 \cdot\eps' \cdot\frac{x^2}{m'} \s{\frac{\binom{\binom{n}{2}}{m-m'}}{\binom{\binom{n}{2}}{m}}}^{1/m'}	\\ 
 & \leq 4^{m/m'} C_0 \cdot \frac{2\delta}{\alpha^{1/2}} \cdot \frac{x^2}{m'} \left(\frac{\frac{\binom{n}{2}!}{(m-m')!\left(\binom{n}{2}-(m-m')\right)!}}{\frac{\binom{n}{2}!}{m!\left(\binom{n}{2}-m\right)!}}\right)^{1/m'}  \\ 
 & \leq 4^{m/m'} C_0 \cdot \frac{2\delta}{\alpha^{1/2}} \cdot \frac{x^2}{m'} \left(\frac{m(m-1)\dots(m-m'+1)}{\left(\binom{n}{2}-m+m'\right)\left(\binom{n}{2}-m+m'-1\right)\dots \left(\binom{n}{2}-m+1\right)}\right)^{1/m'}  \\ 
  &\leq 4^{2/\alpha} C_0 \cdot \frac{2\delta}{\alpha^{1/2}} \cdot \frac{x^2}{m'}\cdot  \frac{m}{\binom{n}{2}-m} 
  \leq  4^{2/\alpha} C_0 \cdot \frac{2\delta}{\alpha^{1/2}} \cdot \frac{m}{m'} \cdot \frac{x^2}{n^2/3} \\ 
  &\leq \delta \cdot \frac{12}{\alpha^{3/2}} \cdot 4^{2/\alpha} \cdot C_0 \leq 1,
\end{align*} 
implying \eqref{e:nxm} as desired. This concludes the proof of the theorem.	
\end{proof}	

\section{Concluding remarks}

We have established bounds on the twin-width of sparse graphs in terms of their maximum average degree (see e.g. \cref{l:mainSparse2})  and have shown that these bounds are essentially tight for random graphs and random regular graphs. In combination with the previously known results,  our bounds, in particular, asymptotically determine $\log \tww(G(n,p))$ for $\frac{1 + \eps(n)}{n} \leq  p \leq 1 - \frac{1 + \eps(n)}{n}$ for some function $\eps(n)$ sufficiently slowly approaching  zero. 

The regime in which $\tww(G(n,p))$ remains least understood is the critical window $p=\frac{1+o(1)}{n}$. A related open question is to find optimal bounds on the twin-width of $n$-vertex graphs with maximum average degree $2 + \eps(n)$, where $\eps(n)$  quickly approaches zero. Such graphs structurally resemble long subdivisions, and the twin-width of subdivisions has been investigated before in ~\cite{bonnet_twin-width_2022-1} and~\cite{ahn_twin-width_2023}. It is tempting to extend our result to this setting, but substantial technical difficulties, in particular, involving the treatment of large degree vertices, remain.  

\bibliographystyle{alphaurl}
\bibliography{refs}

\end{document}